\documentclass{amsart} 
\usepackage{enumerate, amssymb, amsfonts, latexsym}
\usepackage{amscd, amsmath}
\usepackage{graphicx}
\newtheorem{theorem}{Theorem} 
\newtheorem{lemma}{Lemma} 
\newtheorem{definition}{Definition}

\begin{document}

\title{Order of Magnitude of Fourier Coefficients for Almost Periodic Functions}

\author{Alec Train, Rohit Jain, Will Carlson}

\date{\today}

\begin{abstract}
We provide an introduction of some basic facts of uniformly almost periodic functions, such as fourier series representations.  A result is then proved about fourier coefficients which is a generalization of the purely periodic case. We then provide an application of our estimate to the Riemann-Zeta Function. 
\end{abstract}

\maketitle

\section{Introduction}

In recent years there have been many applications for almost periodic functions. They have been studied in the context of signal processing algorithms for detection estimation and classification \cite{N12}, nonuniform sampling \cite{JF95}, dynamical systems, and homogenization theory for elliptic partial differential equations. Our interest in this paper will be to study the order of magnitude of fourier coefficients for almost periodic functions. In this section we present basic results in the theory of uniformly almost periodic functions. In section 2 we recall the construction of the fourier series for almost periodic functions. In section 3 we present our proof in the almost periodic case. Finally in the last section we present an application to analytic number theory.
\\
\\
We start by defining the notion of a uniformly almost periodic function.

\begin{definition} We define the space of uniformly almost periodic functions to be the closure of the set of all trigonometric polynomials $\{\sum\limits_{k=0}^{n} c_{k} e^{ix\lambda_k}\}$ under the sup norm $\|.\|_{\infty}$.
\end{definition}

Here the $\lambda_k$ are arbitrary real numbers, instead of integral multiples of some base frequency.
\\
\\
One of the key theorems of uniformly almost periodic functions is that this definition is equivalent to the following definition:

\begin{definition} A function $f: \mathbb{R}\rightarrow \mathbb{C}$ is uniformly almost periodic if and only if $ \forall \epsilon > 0$, there exists a number $l(\epsilon)$ and a set of translation numbers $T_{f}(\epsilon) = \{\tau_{1}(\epsilon), \tau_{2}(\epsilon), \tau_{3}(\epsilon),... \}$ such that $ \forall \tau \in T_{f}(\epsilon)$, 

$$|f(x+\tau) - f(x)| < \epsilon$$

\end{definition}

We state some basic properties for uniformly almost periodic functions. All proofs can be found in \cite{B55}. We list them here for completeness.

\begin{lemma} (Basic Properties) 
\\
a) Any uniformly almost periodic function is bounded
\\
b) Any uniformly almost periodic function is uniformly continuous
\\
c) Constant multiples, conjugate, sums, and products of uniformly almost periodic functions are uniformly almost periodic
\\
d) If the derivative of a uniformly almost periodic function is uniformly continuous then it is uniformly almost periodic.
\end{lemma}

Our goal is to study the order of magnitude of fourier coefficients for almost periodic functions. Our inspiration are similar estimates for periodic functions. We turn to some definitions that will be useful later.

\begin{definition} The variation of a function $f:\mathbb{R} \to \mathbb{R}$ is 
$$V_{\mathbb{R}}(f) = \sup_{a<b} V_{[a,b]}(f)$$ where $V_{[a,b]}(f) = \sup \{S_{\Gamma} : \Gamma \; \text{is a partition of} \; [a,b]\}$ and $S_{\Gamma} = \sum_{i = 1}^{n} |f(x_{i}) - f(x_{i-1})|$ for any finite partition 
$\Gamma = \{a = x_{0} < ... < x_{n} = b\}$.
\\
\\
We say that a function $f$ has bounded variation if $V_{\mathbb{R}}(f) < \infty$. Furthermore we define the space of bounded variation (BV):
$$BV(\mathbb{R}) = \{f : \mathbb{R} \to \mathbb{R} \; : \text{f has finite total variation on} \; \mathbb{R} \}$$
\end{definition}

We modify the classical definition by imposing a growth estimate on the total variation in order to accomodate uniformly almost periodic functions:

\begin{definition} We define the average total variation of a function $f:\mathbb{R}\rightarrow \mathbb{R}$ to be the following limit:

$$\bar{V}_{\mathbb{R}}(f) = \lim_{T\rightarrow \infty} \frac{1}{T} V_{[0,T]}(f)$$

where $ V_{[0,T]}(f)$ is the total variation of $f$ on the interval $[0,T]$. We say that a function $f$ has average bounded variation if $\bar{V}_{\mathbb{R}}(f) < \infty$. Furthermore we define the space of average bounded variation ($\bar{BV}$):
$$\bar{BV}(\mathbb{R}) = \{f : \mathbb{R} \to \mathbb{R} \; : \text{f has finite average total variation on} \; \mathbb{R} \}$$
\end{definition}

We remark that average total variation agrees with the total variation of a 1-periodic function on compact domains. We now turn to the theorem we would like to generalize \cite{T67},

\begin{theorem} (Taibleson '67) Suppose f is a 1-periodic function on $[0,1]$ which is continuous and of bounded variation. Given the fourier series representation of $f(x)$,

$$f \sim \sum_{j \in Z} c_j e^{i2\pi j x}$$

Then,

$$|c_j| \leq \frac{\|f\|_{BV}}{2\pi j}$$
\end{theorem}

\section{Fourier Series of Almost Periodic Functions}

In this section, we discuss the construction of fourier series for almost periodic functions.  The are three classical results we will describe which involve a notion of mean value of an almost periodic function, a theorem which states that any uniformaly almost periodic function may be represented by a fourier series and provides the construction of the series via a formula for the series coefficients, and finally a result about derivatives of uniformly almost periodic functions.  We refer to the book by C. Corduneanu \cite{C61} for all of these results. We state them here for completeness. 
\begin{lemma} If $f(x)$ is an almost periodic function, then
$$\lim_{T\rightarrow\infty} \int^{a+T}_a f(x) dx = M\{f(x)\}$$
exists uniformly with respect to $a$.
$M\{f(x)\}$ is independent of $a$ and is called the mean value of the almost periodic function $f(x)$.
\end{lemma}
This result defines the mean values $M\{f(x)\}$, and we refer to Corduneanu \cite{C61} for the proof.
Now define $a(\lambda)=M\{f(x)e^{-i\lambda x}\}$.
\begin{lemma} If $f$ is almost periodic, there exists at most a countable set of $\lambda$'s for which $a(\lambda) \neq 0$.
\end{lemma}
The proof we omit as the technical details are not required for what follows.
\\
\\
This result is important because the numbers $\lambda_1, \lambda_2, ..., \lambda_n, ...$ for which $a(\lambda_k) \neq 0$ are called the fourier exponents of the function $f(x)$, and $a(\lambda_k)$ are the fourier coefficients of $f(x)$.  Now we will define $A_k=a(\lambda_k)$ and write the expansion of the function $f$ as follows:
$$f(x)\sim\sum_{k=1}^\infty A_ke^{i\lambda_kx}$$
\\
We may note here that if $f(x)$ is a periodic function, then the fourier series defined this way will coincide with the usual fourier series from the theory of periodic functions.
\begin{lemma} If the derivative (primitive) of an almost periodic function is almost periodic, then its fourier series can be obtained by formal differentiation (integration).
\end{lemma}

\begin{proof}Let $f(x)$ be an almost periodic function with an almost periodic derivative.  Then the mean value $M\{f'(x)e^{-i\lambda x}\}$ exists, and furthermore we claim that $M\{f'(x)e^{i\lambda x}\}=i\lambda M\{f(x)\}$.  To see this, we note that $$\frac{1}{T}\int_a^{a+T}f'(x)e^{-i\lambda x} dx=\frac{1}{T}f(x)e^{-i\lambda x}\rvert_a^{a+T}+i\lambda\frac{1}{T}\int_a^{a+T}f(x)e^{-i\lambda x}dx$$ and then take the limit $T\rightarrow\infty$.
Now from the claim we infer that $f'(x)$ has the same fourier exponents as $f(x)$ except for possibly $\lambda=0$ if it appears as a fourier exponent of $f$.  If we define $A'_k$ to be the fourier coefficients of $f'(x)$, then the relation says that:
$$A'_k=i\lambda_kA_k$$
Therefore $f'(x)\sim\sum_{k=1}^\infty i\lambda_kA_ke^{i\lambda_k x}$.
\\
From this it follows that the primitive $F(x)$ is almost periodic and we have:
$$F(x)=\int_0^xf(t)dt\sim C+\sum_{k=1}^\infty\frac{}{}e^{i\lambda_k x}$$
\end{proof}
A careful remark from Corduneanu's book \cite{C61} points out that the $\lambda=0$ cannot occur among the fourier exponents of an almost periodic function which is the derivative of another almost periodic function, which means this formula is valid even though $\lambda_k$ occur in the denominator.  Note that it is necessary but not sufficient that $\lambda = 0$ does not occur to have that the primitive of an almost periodic function is almost periodic.

\section{The main result}
\label{sec:examples}
We now turn to our contribution.

\begin{theorem} Suppose f is a uniformly almost periodic function on $\mathbb{R}$ which is also of average bounded variation $\|f\|_{\bar{BV}} < \infty$.  Given the fourier series representation for $f(x)$, $$f(x)\sim\sum_{k=1}^\infty A_ke^{i\lambda_kx}$$ 

Then, 

$$|A_j| \leq \frac{\|f\|_{\bar{BV}}}{\lambda_j}$$

\end{theorem}

\begin{proof}

We begin by writing

$$A_j = \lim_{p \rightarrow \infty} \frac{1}{p} \int_0^p f(x) e^{-i\lambda_j x} dx$$

Fix $\epsilon > 0$.  Since $f(x)$ and the exponential function $e^{-i\lambda_{j}x}$ are almost periodic, so is there product, which we denote by $g(x)$.  Hence we can find a set $T_{g}(\epsilon)$ and a positive number $l(\epsilon)$ such that 
$$|g(x+\tau) - g(x)| < \epsilon$$ 

whenever $\tau \in T_{g}(\epsilon)$, and any interval of length $l(\epsilon)$ has non-empty intersection with $T_{g}(\epsilon)$. This follows from taking a common element of $T_{f}(\frac{\epsilon}{32M})$ and $T_{e^{-i\lambda_{j}x}}(\frac{\epsilon}{32M})$ where $M = \|f\|_{L^{\infty}}$.  Furthermore, since we know that the above limit exists, we may take any sequence $p_k \to \infty$. For convenience, we take it so that each $p_k \in T_{g}(\epsilon) \; \forall k$. We consider for a fixed m,

$$\frac{1}{p_m} \int_0^{p_m} f(x) e^{-i\lambda_j x} dx$$

Integrating by parts yields

$$\frac{1}{-i\lambda_k p_m} \left(g(p_k) - g(0) - \int_{0}^{p_m} f'(x)e^{-i\lambda_k x} dx\right) $$

Taking the absolute value of this and recalling that, $ |\lim_{x \to \infty} f(x)| = \lim_{x \to \infty} |f(x)|$,

$$|A_j| \leq \lim_{m \rightarrow \infty} \left(\frac{1}{\lambda_{k} p_{m}} (|g(p_{m}) - g(0)| + \int_0^{p_m} |f'(x)| dx\right)$$

Using the almost periodicity and the linear bounded variation assumption,

$$|A_j| \leq \lim_{m \rightarrow \infty} \left(|\frac{1}{\lambda_k p_m}| \left(\epsilon + V_{[0,p_{m}]}(f)\right)\right)$$

The first term in the above limit converges to $0$ and the second term converges to the average bounded variation. Hence,

$$|A_j| \leq \frac{\|f\|_{\bar{BV}}}{\lambda_j}$$

\end{proof}

The resulting bound also extends to $C^n$ uniformly almost periodic functions under the uniform continuity hypothesis for the derivative of a uniformly almost periodic function $f(x)$. This ensures as observed above that the derivative is also a uniformly periodic function.   

\begin{theorem}
Suppose that $f,f^{'},\dots,f^{n}$ are uniformly almost periodic functions, and that $f^{n}(x)$ has average bounded variation.  Given the fourier series representation for $f(x)$

$$f(x)\sim\sum_{k=1}^\infty A_{k}e^{i\lambda_kx}$$ 

Then, 
$$|A_{j}| \leq \frac{\|f^{n}\|_{\bar{BV}}}{\lambda_{j}^{n+1}}$$
\end{theorem}

\begin{proof} Since we can do term-by-term differentiation as shown above, we have that 

$$f^{n}(x)=\sum_{j \in Z} A_j(i\lambda_j)^n e^{i\lambda_j x}$$

So by Proposition 1, 
$$|A_j(i\lambda_j)^n| \leq \frac{\|f^{n}\|_{\bar{BV}}}{\lambda_j}$$

Hence,

$$|A_j| \leq \frac{\|f^{n}\|_{\bar{BV}}}{\lambda_j^{n+1}}$$

\end{proof}

\section{Application to Analytic Number Theory}

Uniformly almost periodic functions were first introduced by the mathematician and Olympic silver medalist Harald Bohr, brother of Niels Bohr.  Almost periodic functions were originally studied in finite truncations of the Riemann zeta function.  Studying the function through this method gets around using analytic continuation to study the function outside the region for which the usual Dirichlet series is defined. Recall:
 
\begin{definition} (Riemann-Zeta Function) $\zeta(x+iy) = \sum\limits_{n \geq 1} n^{-x}n^{-iy}$
\end{definition}

We rewrite, $\zeta(x+iy) = \sum\limits_{n \geq 1} n^{-x}e^{-iy \log n}$. For a fixed x and a fixed N, we consider the truncated function:
$$\zeta_{x,N}(y) = \sum_{n = 1}^{N} n^{-x} e^{-iy \log n}$$
Since we are now considering an almost periodic function with incommensurable frequencies $\lambda_{n} = \log(n)$.  We can estimate the variation $\|\zeta_{x,N}(y)\|_{\bar{BV}}$. In particular,
\begin{theorem}
Define $\zeta_{x,N}(y)^{J}$ to be the $J$-th derivative of the partial sum. Hence for all J, we get:
$$\|\zeta_{1/2,N}(y)^{J}\|_{\bar{BV}} \geq \max\limits_{(1,...,N)} n^{-1/2} (\log n)^{J+1}$$
\end{theorem}

\begin{proof}
From the previous theorem we obtain $\forall J$,
$$n^{-x} \leq \frac{1}{(\log n)^{J+1}} \|\zeta_{x,N}(y)^{J}\|_{\bar{BV}}.$$ 
Considering $x=\frac{1}{2}$, we get the lower estimate:
$$\|\zeta_{1/2,N}(y)^{J}\|_{\bar{BV}} \geq \max\limits_{(1,...,N)} \frac{(\log n)^{J+1}}{\sqrt{n}}.$$
\end{proof}

\end{document}